\documentclass[runningheads,a4paper]{llncs}

\usepackage[latin1]{inputenc}
\usepackage{amsfonts}
\usepackage{graphicx}
\usepackage{graphicx}
\usepackage{epsfig}
\usepackage{amssymb,amsmath,color,theorem}
\usepackage{amsmath}
\usepackage{amssymb}
\usepackage{subfigure}

\begin{document}

\title{Convergence Rate of the Causal Jacobi Derivative Estimator}
\titlerunning{Convergence Rate of the Causal Jacobi Derivative Estimator}

\author{Da-yan Liu\inst{1,}\inst{2}  \and Olivier Gibaru\inst{1,}\inst{3} \and Wilfrid Perruquetti\inst{1,}\inst{4}}
\authorrunning{D.Y. Liu, O. Gibaru, and W.~Perruquetti}

\institute{INRIA Lille-Nord Europe, Équipe Projet Non-A,  Parc
Scientifique de la Haute Borne 40, avenue Halley B{\^{a}}t. A, Park
Plaza, 59650 Villeneuve d'Ascq, France \and Université  de Lille 1,
Laboratoire de Paul Painlevé,   59650, Villeneuve d'Ascq, France
\\ \email{dayan.liu@inria.fr} \and  Arts et Metiers ParisTech centre de
Lille, Laboratory of Applied Mathematics and Metrology (L2MA),  8 Boulevard Louis XIV, 59046 Lille Cedex, France
\\ \email{olivier.gibaru@ensam.eu} \and \'{E}cole Centrale de Lille, Laboratoire de LAGIS,
 BP 48, Cit\'e Scientifique, 59650
Villeneuve d'Ascq, France\\ \email{wilfrid.perruquetti@inria.fr}}

\maketitle

\begin{abstract}
Numerical causal derivative estimators from noisy data are essential
for real time applications especially for control applications or
fluid simulation so as to address the new paradigms in solid
modeling and video compression. By using an analytical point of view
due to Lanczos \cite{C. Lanczos} to this causal case, we revisit
$n^{th}$\ order derivative estimators originally introduced within
an algebraic framework by Mboup, Fliess and Join in \cite{num,num0}.
Thanks to a given noise level $\delta$ and a well-suitable
integration length window, we show that the derivative estimator
error can be $\mathcal{O}(\delta ^{\frac{q+1}{n+1+q}})$ where $q$\
is the order of truncation of the Jacobi polynomial series expansion
used. This so obtained bound helps us to choose the values of our
parameter estimators. We show the efficiency of our method on some
examples. 
\end{abstract}


\section{Introduction}

There exists a large class of numerical derivative estimators which
were introduced according to different scopes
(\cite{R19,R29,{J. Cheng},{Z. Wang2},{A.G. Ramm},{D.A. Murio2},{G. Nakamura}}).
When the initial discrete
data are corrupted by a noise, numerical differentiation becomes an
ill-posed problem.\ By using an algebraic method inspired by
\cite{mexico,mboup,Med08}, Mboup, Fliess and Join introduced in
\cite{num,num0} real-time numerical differentiation by integration
estimators that provide an effective response to this problem.\
Concerning the robustness of this method, \cite{ans,shannon} give
more  theoretical foundations. These estimators extend those
introduced by \cite{C. Lanczos,S.K. Rangarajana,Z. Wang} in the
sense that they use Jacobi polynomials. In \cite{num}, the authors
show that the mismodelling due to the truncation of the Jacobi
expansion can be improved by allowing a small time-delay in the
derivative estimation. This time-delay is obtained as the product of
the length of the integration window by the smallest root of the
first Jacobi polynomial in the remainder of series expansion.

In \cite{Num algo}, we extend to the real domain the parameter values of these
Jacobi estimators.\ This allows us to decrease the value of this smallest root
and consequently the time-delay estimation. In \cite{Jcam}, we study for
center derivative Jacobi estimators the convergence rate of these estimators.

Thanks to these results, we propose in this article to tackle the
causal convergence rate case. We give an optimal convergence rate of
these estimators depending on the derivative order, the noise level
of the data and the truncation order. Moreover, we show that the
estimators for the $n^{th}$ order derivative of a smooth function
can be obtained by taking $n$ derivations to the zero-order
estimator of the function. Hence, we can give a simple expression
for these estimators, which is much easier to calculate than the one
given in \cite{Jcam}.

This paper is organized as follows: in Section \ref{section2} the causal
estimators introduced in \cite{num} are studied with extended parameters. The
convergence rate of these estimators are then studied. Finally, numerical
tests are given in Section~\ref{section3}. They help us to show the efficiency
and the stability of this proposed estimators. We will see that the numerical
integration error may also reduce this time-delay for a special class of functions.

\section{Derivative Estimations by Using Jacobi Orthogonal Series}

\label{section2}

Let $f^{\delta}=f+\varpi$ be a noisy function defined on an open
interval $I\subset\mathbb{R}$, where $f\in \mathcal{C}^{n}(I)$ with
$n\in\mathbb{N}$ and $\varpi$ be a noise\footnote{More generally,
the noise is a stochastic process, which is bounded with certain
probability and integrable in the sense of convergence in mean
square (see \cite{Num algo}).} which is bounded and integrable with
a noise level $\delta$, $i.e.$ $\delta=\displaystyle\sup_{x\in
I}|\varpi(x)|$. Contrary to \cite{S.K. Rangarajana} where Legendre
polynomials were used, we propose to use, as in \cite{num,num0},
truncated Jacobi orthogonal series so as to estimate the $n^{th}$
order derivative of $f$. In this section, we are going to give a
family of causal Jacobi estimators by using Jacobi polynomials
defined on $[0,1]$. From now on, we assume that the parameter $h>0$\
and we denote $I_{h}:=\left\{  x\in I;\left[  x-h,x\right]  \subset
I\right\}  $.

The $n^{th}$ order Jacobi polynomials (see \cite{G. Szego}) defined
on $[0,1]$ are defined  as follows
\begin{equation}
{P}_{n}^{(\alpha,\beta)}(t)=\displaystyle\sum_{j=0}^{n}\binom{n+\alpha}%
{j}\binom{n+\beta}{n-j}\left(  {t-1}\right)  ^{n-j}\left(  t\right)
^{j}\label{jacobi}%
\end{equation}
where $\alpha,\beta\in]-1,+\infty\lbrack$. Let $g_{1}$ and $\,g_{2}$
be two functions which belong to $\mathcal{C}([0,1])$, then we
define the scalar product of these functions by $$\left\langle
g_{1}(\cdot),g_{2}(\cdot)\right\rangle
_{\alpha,\beta}:=\int_{0}^{1}{w}_{\alpha,\beta}(t)g_{1}(t)g_{2}(t)dt,$$
where ${w}_{\alpha,\beta}(t)=(1-t)^{\alpha}t^{\beta}$ is a weighted
function. Hence, we can denote its associated norm by
$\Vert\cdot\Vert_{\alpha,\beta}$. We then have
\begin{equation}
\Vert{P}_{n}^{\left(  \alpha,\beta\right)  }\Vert_{\alpha,\beta}^{2}=\frac
{1}{2n+\alpha+\beta+1}\frac{\Gamma(\alpha+n+1)\Gamma(\beta+n+1)}{\Gamma
(\alpha+\beta+n+1)\Gamma(n+1)}.\label{norm2}%
\end{equation}
Let us recall two useful formulae (see \cite{G. Szego})
\begin{align}
{P}_{n}^{(\alpha,\beta)}(t){w}_{\alpha
,\beta}(t)&=\frac{(-1)^{n}}{n!}\frac{d^{n}}{dt^{n}}[{w}_{\alpha+n,\beta
+n}(t)] \quad (\text{the Rodrigues formula}),\label{rodrigues2}\\
\frac{d}{dt}[{P}_{n}^{(\alpha,\beta)}(t)]&=(n+\alpha+\beta
+1){P}_{n-1}^{(\alpha+1,\beta+1)}(t).\label{derivation}%
\end{align}

Let us ignore the noise $\varpi$ for a moment. Since $f$ is assumed
to belong to $\mathcal{C}^{n}(I)$, we define the $q^{th}$ ($q
\in\mathbb{N}$) order truncated Jacobi orthogonal series of
$f^{(n)}(x-ht)$ $(t \in[0,1])$ by the following operator: $\forall
x\in{I}_{h},$
\begin{equation}
\label{orthogonal_series2}{D}_{h,\alpha,\beta,q}^{(n)}f(x-th):=\sum_{i=0}%
^{q}\frac{\left\langle {P}_{i}^{(\alpha+n,\beta+n)}(\cdot),f^{(n)}(x-h
\cdot)\right\rangle _{\alpha+n,\beta+n}}{\Vert{P}_{i}^{(\alpha+n,\beta
+n)}\Vert^{2}_{\alpha+n,\beta+n}}\ {P}_{i}^{(\alpha+n,\beta+n)}(t).
\end{equation}
We also define the $(q+n)^{th}$ order truncated Jacobi orthogonal series of
$f(x-ht)$ $(t \in[0,1])$ by the following operator
\begin{equation}
\label{orthogonal2}\forall x\in{I}_{h},\ {D}^{(0)}_{h,\alpha,\beta
,q}f(x-th):=\sum_{i=0}^{q+n}\frac{\left\langle {P}_{i}^{(\alpha,\beta)}%
(\cdot),f(x-h \cdot)\right\rangle _{\alpha,\beta}}{\Vert{P}_{i}^{(\alpha
,\beta)}\Vert^{2}_{\alpha,\beta}}\ {P}_{i}^{(\alpha,\beta)}(t).
\end{equation}

It is easy to show that for each fixed value $x$, $D_{h,\alpha,\beta,q}%
^{(0)}f(x-h\cdot)$ is a polynomial which approximates the function
$f(x-h\cdot)$. We can see in the following lemma that
${D}_{h,\alpha,\beta ,q}^{(n)}f(x-h\cdot)$ is in fact connected to
the $n^{th}$ order derivative of
${D}_{h,\alpha,\beta,q}^{(0)}f(x-h\cdot)$. It can be expressed as an
integral of $f$.

\begin{lemma}
\label{proposition5} Let $f\in \mathcal{C}^{n}(I)$, then we have
\begin{equation}
\label{D3}\forall x\in{I}_{h},\ {D}_{h,\alpha,\beta,q}^{(n)}f(x-th)=
\frac{1}{(-h)^{n}} \frac {d^{n}}{d t^{n}}\left[
{D}^{(0)}_{h,\alpha,\beta,q}f(x-th)\right].
\end{equation}
Moreover, we have
\begin{equation}
\label{D_h_N3}\forall x\in{I}_{h},\ {D}_{h,\alpha,\beta,q}^{(n)}%
f(x-th)=\frac{1}{(-h)^{n} } \int_{0}^{1} {Q}_{\alpha,\beta,n,q,t}(\tau)
f(x-h\tau) d\tau,
\end{equation}
where ${Q}_{\alpha,\beta,n,q,t}(\tau) = {w}_{\alpha,\beta}(\tau)\displaystyle\sum_{i=0}^{q}%
C_{\alpha,\beta,n,i} P_{i}^{(\alpha+n,\beta+n)}(t)
{P}_{n+i}^{(\alpha,\beta)}(\tau)$, and

\noindent
$C_{\alpha,\beta,n,i}=\frac{(\beta+\kappa+2n+2i+1)\Gamma(\beta+\alpha+2n+i+1)\Gamma(n+i+1)}{\Gamma(\beta+n+i+1)\Gamma(\alpha+n+i+1)}
$ with $\alpha, \beta \in ]-1,+\infty[$.
\end{lemma}

\begin{proof}
By applying  $n$ times derivations to (\ref{orthogonal2}) and by
using (\ref{derivation}), we obtain
\begin{equation}\label{9}
\begin{split}
&  \frac{d^{n}}{dt^{n}}\left[  {D}_{h,\alpha,\beta,q}^{(0)}f(x-th)\right]  \\
= & \sum_{i=0}^{q}\frac{\left\langle {P}_{i+n}%
^{(\alpha,\beta)}(\cdot),f(x-h\cdot)\right\rangle _{\alpha,\beta}}{\Vert
{P}_{i+n}^{(\alpha,\beta)}\Vert_{\alpha,\beta}^{2}}\frac{d^{n}}{dt^{n}}\left[
{P}_{i+n}^{(\alpha,\beta)}(t)\right]  \\
= & \sum_{i=0}^{q}\frac{\left\langle {P}_{i+n}%
^{(\alpha,\beta)}(\cdot),f(x-h\cdot)\right\rangle _{\alpha,\beta}}{\Vert
{P}_{i+n}^{(\alpha,\beta)}\Vert_{\alpha,\beta}^{2}}\frac{\Gamma(\alpha
+\beta+2n+i+1)}{\Gamma(\alpha+\beta+n+i+1)}{P}_{i}^{(\alpha+n,\beta+n)}(t).
\end{split}
\end{equation}
By applying two times the Rodrigues formula given in (\ref{rodrigues2}) and by
taking $n$ integrations by parts, we get
\[%
\begin{split}
&  \left\langle {P}_{i}^{(\alpha+n,\beta+n)}(\cdot),f^{(n)}(x-h\cdot
)\right\rangle _{\alpha+n,\beta+n}\\
= &  \int_{0}^{1}{w}_{\alpha+n,\beta+n}(\tau){P}_{i}^{(\alpha+n,\beta+n)}%
(\tau)\,f^{(n)}(x-h\tau)d\tau\\
= &  \int_{0}^{1}\frac{(-1)^{i}}{i!}{w}_{\alpha+n+i,\beta+n+i}^{(i)}%
(\tau)\,f^{(n)}(x-h\tau)d\tau\\
= &  \frac{1}{(-h)^{n}}\int_{0}^{1}\frac{(-1)^{i+n}}{i!}{w}_{\alpha
+n+i,\beta+n+i}^{(n+i)}(\tau)\,f(x-h\tau)d\tau\\
= &  \frac{1}{(-h)^{n}}\int_{0}^{1}\frac{(n+i)!}{i!}{w}_{\alpha,\beta}%
(\tau){P}_{n+i}^{(\alpha,\beta)}(\tau)\,f(x-h\tau)d\tau.
\end{split}
\]
Then, after some calculations by using (\ref{norm2}) we can obtain
\begin{equation}%
\begin{split}
&  \frac{\left\langle {P}_{i}^{(\alpha+n,\beta+n)}(\cdot),f^{(n)}%
(x-h\cdot)\right\rangle _{\alpha+n,\beta+n}}{\Vert{P}_{i}^{(\alpha+n,\beta
+n)}\Vert_{\alpha+n,\beta+n}^{2}}\\
= &  \frac{\left\langle {P}_{i+n}^{(\alpha,\beta)}(\cdot),f(x-h\cdot
)\right\rangle _{\alpha,\beta}}{(-h)^{n}\Vert{P}_{n+i}^{(\alpha,\beta)}%
\Vert_{\alpha,\beta}^{2}}\frac{\Gamma(\alpha+\beta+2n+i+1)}{\Gamma
(\alpha+\beta+n+i+1)}.
\end{split}
\end{equation}
Finally, by taking (\ref{orthogonal_series2}) and (\ref{9}) we obtain
\begin{equation}
\forall x\in{I}_{h},\ {D}_{h,\alpha,\beta,q}^{(n)}f(x-th)= \frac{1}{(-h)^{n} }\frac{d^{n}}{dt^{n}%
}\left[  {D}_{h,\alpha,\beta,q}^{(0)}f(x-th)\right]  .
\end{equation}
The proof is complete. \qed
\end{proof}

If we consider the noisy function $f^{\delta}$, then it is sufficient to
replace $f(x-h\cdot)$ in $(\ref{D_h_N3})$ by $f^{\delta}(x-h\cdot)$. In
\cite{num}, for a given value $t_{\tau} \in [0,1]$, ${D}_{h,\alpha,\beta,q}%
^{(n)}f^{\delta}(x-t_{\tau}h)$ (with $\alpha,\beta\in\mathbb{N}$ and
$q\leq\alpha+n$) was proposed as a point-wise estimate of
$f^{(n)}(x)$ by admitting a time-delay $t_{\tau}h$. We assume here
that these values $\alpha$ and $\beta$ belong to $\left]
-1,+\infty\right[$. This is possible due to the definition of the
Jacobi polynomials. Contrary to \cite{num}, we do not have
constraints on the value of the truncation order $q$. Moreover, the
function ${Q}_{\alpha,\beta,n,q,t}$ is  easier to calculate than the
one given in \cite{Jcam}. Thus, we can define the extended
point-wise estimators as follows.

\begin{definition}
Let $f^{\delta}=f+\varpi$ be a noisy function, where $f\in C^{n}(I)$
and $\varpi$ be a bounded and integrable noise with a noise level
$\delta$. Then a family of causal Jacobi estimators of $f^{(n)}$ is
defined as
\begin{equation}
\forall x\in{I}_{h},\ {D}_{h,\alpha,\beta,q}^{(n)}f^{\delta}(x-t_{\tau
}h)=\frac{1}{(-h)^{n}}\int_{0}^{1}{Q}_{\alpha,\beta,n,q,t_{\tau}}%
(u)f^{\delta}(x-h u)d u,\label{D_h_N2}%
\end{equation}
where $\alpha,\beta\in]-1,+\infty\lbrack$, $q,n\in\mathbb{N}$ and $t_{\tau}$
is a fixed value on $[0,1]$.
\end{definition}

Hence, the estimation error comes from two sources : the remainder
terms in the Jacobi series expansion of $f^{(n)}(x-h\cdot)$ and the
noise part. In the following proposition, we study these estimation
errors.

\begin{proposition}
\label{proposition6} Let $f^{\delta}$ be a noisy function where $f
\in \mathcal{C}^{n+1+q}(I)$ and $\varpi$ be a bounded and integrable
noise with a noise level $\delta$.  Assume that there exists
${M}_{n+1+q}>0$ such that for any $x\in I$,
$\left|f^{(n+q+1)}(x)\right|\leq {M}_{n+1+q}$, then
\begin{equation}
\left\Vert
{D}_{h,\alpha,\beta,{q}}^{(n)}f^{\delta}(x-t_{\tau}h)-f^{(n)}(x-t_{\tau}h)\right\Vert
_{\infty}\leq C_{q,t_{\tau}}h^{q+1}+E_{q,t_{\tau}}\frac{\delta}{h^{n}},
\end{equation}
where $C_{q,t_{\tau}}=M_{n+1+q}\left(\frac{1}{(n+1+q)!}\int_{0}^{1}
\left|u^{n+1+q} {Q}_{\alpha,\beta,n,q,t_{\tau}}(u)\right| \,
d u+ \frac{t_{\tau}^{q+1}}{(q+1)!}\right)$ and \\
\noindent$E_{q,t_{\tau}}=\int
_{0}^{1}\left|{Q}_{\alpha,\beta,n,q,t_{\tau}}(u)\right|\,d u.$
Moreover, if we choose ${h}=\left[
\frac{nE_{q,t_{\tau}}}{(q+1)C_{q,t_{\tau}}}\delta\right]
^{\frac{1}{n+q+1}}$, then we have
\begin{equation}
\left\Vert
{D}_{h,\alpha,\beta,{q}}^{(n)}f^{\delta}(x-t_{\tau}h)-f^{(n)}(x-t_{\tau}h)\right\Vert
_{\infty}=\mathcal{O}(\delta^{\frac{q+1}{n+1+q}}).
\end{equation}
\end{proposition}

\begin{proof}
By taking the Taylor series expansion  of $f$ at $x$, we then have
for any $x\in{I}_{h}$ that there exists ${\xi}\in]x-h,x[$ such that
\begin{equation}
f(x-t_{\tau}h)=f_{n+q}(x-t_{\tau}h)+\frac{(-h)^{n+1+q}t_{\tau}^{n+1+q}%
}{(n+1+q)!}f^{(n+1+q)}({\xi}),\label{Taylor3}%
\end{equation}
where $f_{n+q}(x-t_{\tau}h)=\displaystyle\sum_{j=0}^{n+q}\frac{(-h)^{j}%
t_{\tau}^{j}}{j!}f^{(j)}(x)$ is the $(n+q)^{th}$ order truncated Taylor series
expansion of $f(x-t_{\tau}h)$. By using $(\ref{D_h_N3})$ with $f_{n+q}%
(x-h\cdot)$ we obtain
\begin{equation}
f_{n+q}^{(n)}(x-t_{\tau}h)=\frac{1}{(-h)^{n}}\int_{0}^{1}{Q}_{\alpha
,\beta,n,q,t_{\tau}}(\tau)f_{n+q}(x-h\tau)d\tau.\label{f_N2}%
\end{equation}
Thus, by using $(\ref{D_h_N2})$ and $(\ref{f_N2})$ we obtain
\[%
\begin{split}
&  {D}_{h,\alpha,\beta,{q}}^{(n)}f(x-t_{\tau}h)-f_{n+q}^{(n)}(x-t_{\tau}h)\\
= &  \frac{(-h)^{q+1}}{(n+1+q)!}\int_{0}^{1}{Q}_{\alpha,\beta,n,q,t_{\tau}%
}(\tau)\tau^{n+1+q}f^{(n+1+q)}({\xi})d\tau.
\end{split}
\]
Consequently, if for any $x\in I$ $\left\vert f^{(n+1+q)}(x)\right\vert \leq
M_{n+1+q}$, then by taking the $q^{th}$ order truncated Taylor series
expansion of $f^{(n)}(x-t_{\tau}h)$
\[
f^{(n)}(x-t_{\tau}h)=f_{n+q}^{(n)}(x-t_{\tau}h)+\frac{(-h)^{1+q}t_{\tau}%
^{1+q}}{(1+q)!}f^{(n+1+q)}(\hat{\xi}),
\]
we have
\begin{equation} \label{bias2}
\begin{split}
&  \left\Vert {D}_{h,\alpha,\beta,{q}}^{(n)}f(x-t_{\tau}h)-f^{(n)}(x-t_{\tau
}h)\right\Vert _{\infty}\\
\leq &  \left\Vert {D}_{h,\alpha,\beta,{q}}^{(n)}f(x-t_{\tau}h)-f_{n+q}%
^{(n)}(x-t_{\tau}h)\right\Vert _{\infty}+\left\Vert f_{n+q}^{(n)}(x-t_{\tau
}h)-f^{(n)}(x-t_{\tau}h)\right\Vert _{\infty}\\
\leq &  h^{q+1}M_{n+1+q}\left(  \frac{1}{(n+1+q)!}\int_{0}^{1}\left\vert
\tau^{n+1+q}{Q}_{\alpha,\beta,n,q,t_{\tau}}(\tau)\right\vert \,d\tau
+\frac{t_{\tau}^{q+1}}{(q+1)!}\right).
\end{split}
\end{equation}
Since
\[%
\begin{split}
&  \left\Vert {D}_{h,\alpha,\beta,{q}}^{(n)}f^{\delta}(x-t_{\tau}%
h)-{D}_{h,\alpha,\beta,{q}}^{(n)}f(x-t_{\tau}h)\right\Vert _{\infty}\\
& \hskip 3pc~=  \left\Vert {D}_{h,\alpha,\beta,{q}}^{(n)}\left[  f^{\delta}(x-t_{\tau
}h)-f(x-t_{\tau}h)\right]  \right\Vert _{\infty}\\
&  \hskip 3pc\leq   \frac{\delta}{h^{n}}\int_{0}^{1}\left\vert Q_{\alpha,\beta,n,q}%
(\tau)\right\vert \,d\tau,
\end{split}
\]
by using (\ref{bias2}) we get
\[%
\begin{split}
&  \left\Vert {D}_{h,\alpha,\beta,{q}}^{(n)}f^{\delta}(x-t_{\tau}%
h)-f^{(n)}(x-t_{\tau}h)\right\Vert _{\infty}\\
& \hskip 3pc \leq   \left\Vert {D}_{h,\alpha,\beta,{q}}^{(n)}f^{\delta}(x-t_{\tau}%
h)-{D}_{h,\alpha,\beta,{q}}^{(n)}f(x-t_{\tau}h)\right\Vert _{\infty}\\
& \hskip 5pc  +\left\Vert {D}_{h,\alpha,\beta,{q}}^{(n)}f(x-t_{\tau}h)-f^{(n)}(x-t_{\tau
}h)\right\Vert _{\infty}\\
& \hskip 3pc \leq   C_{q,t_{\tau}}h^{q+1}+E_{q,t_{\tau}}\frac{\delta}{h^{n}},
\end{split}
\]
where $C_{q,t_{\tau}}=M_{n+1+q}\left(  \frac{1}{(n+1+q)!}\int_{0}%
^{1}\left\vert \tau^{n+1+q}{Q}_{\alpha,\beta,n,q,t_{\tau}}(\tau)\right\vert
\,d\tau+\frac{t_{\tau}^{q+1}}{(q+1)!}\right)  $ and \newline\noindent
$E_{q,t_{\tau}}=\int_{0}^{1}\left\vert {Q}_{\alpha,\beta,n,q,t_{\tau}}%
(\tau)\right\vert \,d\tau.$

Let us denote the error bound by $\psi(h)=C_{q,t_{\tau}}h^{q+1}+E_{q,t_{\tau}%
}\frac{\delta}{h^{n}}$. Consequently, we can calculate its minimum value. It
is obtained for ${h}^{*}=\left[  \frac{nE_{q,t_{\tau}}}{(q+1)C_{q,t_{\tau}}%
}\delta\right] ^{\frac{1}{n+q+1}}$ and
\begin{equation}
\psi({h}^{*})=\frac{n+1+q}{q+1}\left(  \frac{q+1}{n}\right)  ^{\frac{n}%
{n+1+q}}C_{q,t_{\tau}}^{\frac{n}{n+1+q}}E_{q,t_{\tau}}^{\frac{q+1}{n+1+q}%
}\delta^{\frac{q+1}{n+1+q}}.
\end{equation}
The proof is complete. \qed
\end{proof}

Let us mention that if we set $t_{\tau}=\theta_{q+1}$, the smallest root of
the Jacobi polynomial ${P}_{q+1}^{(\alpha+n,\beta+n)}$ in
(\ref{orthogonal_series2}), then ${D}_{h,\alpha,\beta,q}^{(n)}f(x-\theta
_{q+1}h)$ becomes the $(q+1)^{th}$ order truncated Jacobi orthogonal series of
$f^{(n)}(x-\theta_{q+1}h)$. Hence, we have the following corollary.

\begin{corollary}
\label{corollary5} Let $f\in \mathcal{C}^{n+2+q}(I)$ where $q$ is an
integer. If we set $t_{\tau}=\theta_{q+1}$, the smallest root of the
Jacobi  polynomial ${P}_{q+1}^{(\alpha+n,\beta+n)}$ in
(\ref{D_h_N2}) and we assume that there
exists ${M}_{n+2+q}>0$ such that for any $x\in I$, $\left\vert f^{(n+q+2)}%
(x)\right\vert \leq{M}_{n+2+q}$, then we have
\[
\left\Vert {D}_{h,\alpha,\beta,{q}}^{(n)}f^{\delta}(x-\theta_{q+1}%
h)-f^{(n)}(x-\theta_{q+1}h)\right\Vert _{\infty}\leq C_{q,\theta_{q+1}}%
h^{q+2}+E_{q,\theta_{q+1}}\frac{\delta}{h^{n}},
\]
where $C_{q,\theta_{q+1}}={M}_{n+2+q}\left(  \frac{1}{(n+q+2)!}\int_{0}%
^{1}|\tau^{n+q+2}{Q}_{\alpha,\beta,n,q,\theta_{q+1}}(\tau)|\,d\tau
+\frac{t^{q+2}}{(q+2)!}\right)  $ and $E_{q,\theta_{q+1}}$ is given
in Proposition \ref{proposition6}.
Moreover, if we choose $\hat{{h}}=\left[  \frac{nE_{q,\theta_{q+1}}}%
{(q+2)C_{q,\theta_{q+1}}}\delta\right]  ^{\frac{1}{n+q+2}}$, then we
have
\[
\left\Vert {D}_{h,\alpha,\alpha,{q}}^{(n)}f^{\delta}(x-\theta_{q+1}%
h)-f^{(n)}(x-\theta_{q+1}h)\right\Vert _{\infty}=\mathcal{O}(\delta^{\frac{q+2}{n+2+q}%
}).
\]

\end{corollary}

\begin{proof}
If $t_{\tau}=\theta_{q+1}$, the smallest root of polynomial ${P}%
_{q+1}^{(\alpha+n,\beta+n)}$ in (\ref{orthogonal_series2}), then we
have
\[ \begin{split}
&{D}_{h,\alpha,\beta,q}^{(n)}f(x-\theta_{q+1}h) \\
&\hskip 5pc =\sum_{i=0}^{q+1}\frac
{\left\langle {P}_{i}^{(\alpha+n,\beta+n)}(\cdot),f^{(n)}(x-h\cdot
)\right\rangle _{\alpha+n,\beta+n}}{\Vert{P}_{i}^{(\alpha+n,\beta+n)}%
\Vert_{\alpha+n,\beta+n}^{2}}\
{P}_{i}^{(\alpha+n,\beta+n)}(\theta_{q+1}).
\end{split}\]
This proof can be completed by taking the $(n+q+1)^{th}$ order truncated
Taylor series expansion of $f$ as it was done in Proposition
\ref{proposition6}. \qed
\end{proof}


The numerical calculation of $E_{q,\theta_{q+1}}$ for $q$,
$n\in\mathbb{N}$ and $\alpha,\beta\in]-1,10]$ shows that
$E_{q,\theta_{q+1}}$ increases with respect to $q$. Hence, in order
to reduce the noise influence it is preferable to choose $q$ as
small as possible. However, $C_{q,\theta_{q+1}}$ decreases with
respect to $q$. A compromise consists in choosing $q=2$. If we take
$D_{h,\alpha ,\beta,{2}}^{(n)}f^{\delta}(x-\theta_{2}h)$ as an
estimator of $f^{(n)}(x)$, then we produce a time-delay
$\theta_{2}h$. For this choice of $\theta_{2}$, we have
$D_{h,\alpha,\beta,{1}}^{(n)}f^{\delta}(x-\theta_{2}h)=D_{h,\alpha
,\beta,{2}}^{(n)}f^{\delta}(x-\theta_{2}h)$. We can see that the
estimators $D_{h,\alpha ,\beta,{2}}^{(n)}f^{\delta}(x)$ do not
produce a time-delay but $C_{1,\theta _{2}}<C_{2,0}$ and
$E_{1,\theta_{2}}<E_{2,0}$. This generally introduces more important
estimation errors. Consequently, so as to
estimate $f^{(n)}(x)$ we use $D_{h,\alpha,\beta,{1}%
}^{(n)}f^{\delta}(x-\theta_{2}h)$ which presents a time-delay.

\section{Numerical Experiments}

\label{section3} In order to show the efficiency and the stability
of the previously proposed estimators, we give some numerical
results in this section.

From now on, we assume that
$f^{\delta}(x_{i})=f(x_{i})+c\varpi(x_{i})$ with $x_{i}=T_{s}i$ for
$i=0,\cdots,500$ ($T_{s}=\frac{1}{100}$), is a noisy measurement of
$f(x)=\exp(\frac{-x}{1.2})\sin(6x+\pi)$. The noise $c\varpi(x_{i})$
is simulated from a zero-mean white Gaussian $iid$ sequence by using
the Matlab function 'randn' with STATE reset to $0$. Coefficient $c$
is adjusted in such a way that the signal-to-noise ratio $SNR=10\log
_{10}\left(
\frac{\sum|y(t_{i})|^{2}}{\sum|c\varpi(t_{i})|^{2}}\right)  $ is
equal to $SNR=22.2\text{dB}$ (see, e.g., \cite{R17} for this well
known concept in signal processing). By using the well known
three-sigma rule, we can assume that the noise level for $c\varpi$
is equal to $3c$. We can see the noisy signal in Figure
\ref{figure_signal}. We use the trapezoidal method in order to
approximate the integrals in our estimators where we use $m+1$
discrete values. The estimated derivatives of $f$ at $x_{i}\in
I=[0,5]$ are calculated from the noise data $f^{\delta}(x_{j})$ with
$x_{j}\in\lbrack -x_{i}-h,x_{i}]$ where $h=mT_{s}$.

\begin{figure}
\centering {\includegraphics[scale=0.6]{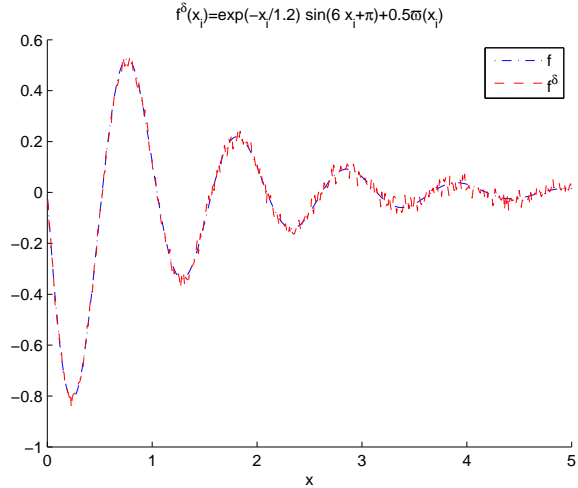}}\caption{ Signal
and noisy
signal.}%
\label{figure_signal}%
\end{figure}

We can see the estimation results for the first order derivative of
$f$ in Figure \ref{simulation0}. The corresponding estimation errors
are given in Figure \ref{simulation1} and in Figure
\ref{simulation2}. We can see that the estimate given by the causal
Jacobi estimator with integer parameters introduced in \cite{num}
(dash line), produces a time-delay of value $\theta_{q+1}h=0.11$.
The estimate given by the causal Jacobi estimator with extended
parameters (dotted line) is time-delay free. Firstly, the root
values $\theta_{q+1}$ for $q=0,1$ can be reduced with the extended
negative parameters, so does the time-delay. Secondly, the numerical
integration method with a negative value for $\beta$ produces a
numerical error which allows us to finally compensate this reduced
time-delay. This last phenomena is due to the fact that the initial
function $f$ satisfies the following differential equation
$f^{(2)}+kf=g$ where $k\in\mathbb{R}$ and $g$ is a continuous
function. Consequently, in the case of  the first order derivative
estimations, we can verify that this numerical error which depends
on $f$ may reduce the effect of the error due to the truncation in
the Jacobi series expansion.  This is due to the fact that the
truncation error depends on $f^{(2)}$. Hence, the final total error
is $\mathcal{O}(g)$. Finally, since
$D_{mT_{s},\alpha,\beta,q}^{(1)}f^{\delta}(x_{i}-\theta_{2}h)$
produces a
time-delay of value $\theta_{2}h$, we give  in Figure \ref{simulation3} the errors
$D_{mT_{s}%
,\alpha,\beta,q}^{(1)}f^{\delta}(x_{i}-\theta_{2}h)-f^{(1)}(x_{i}-\theta
_{2}h)$.

\begin{figure}[h!]
\centering
{\includegraphics[scale=0.6]{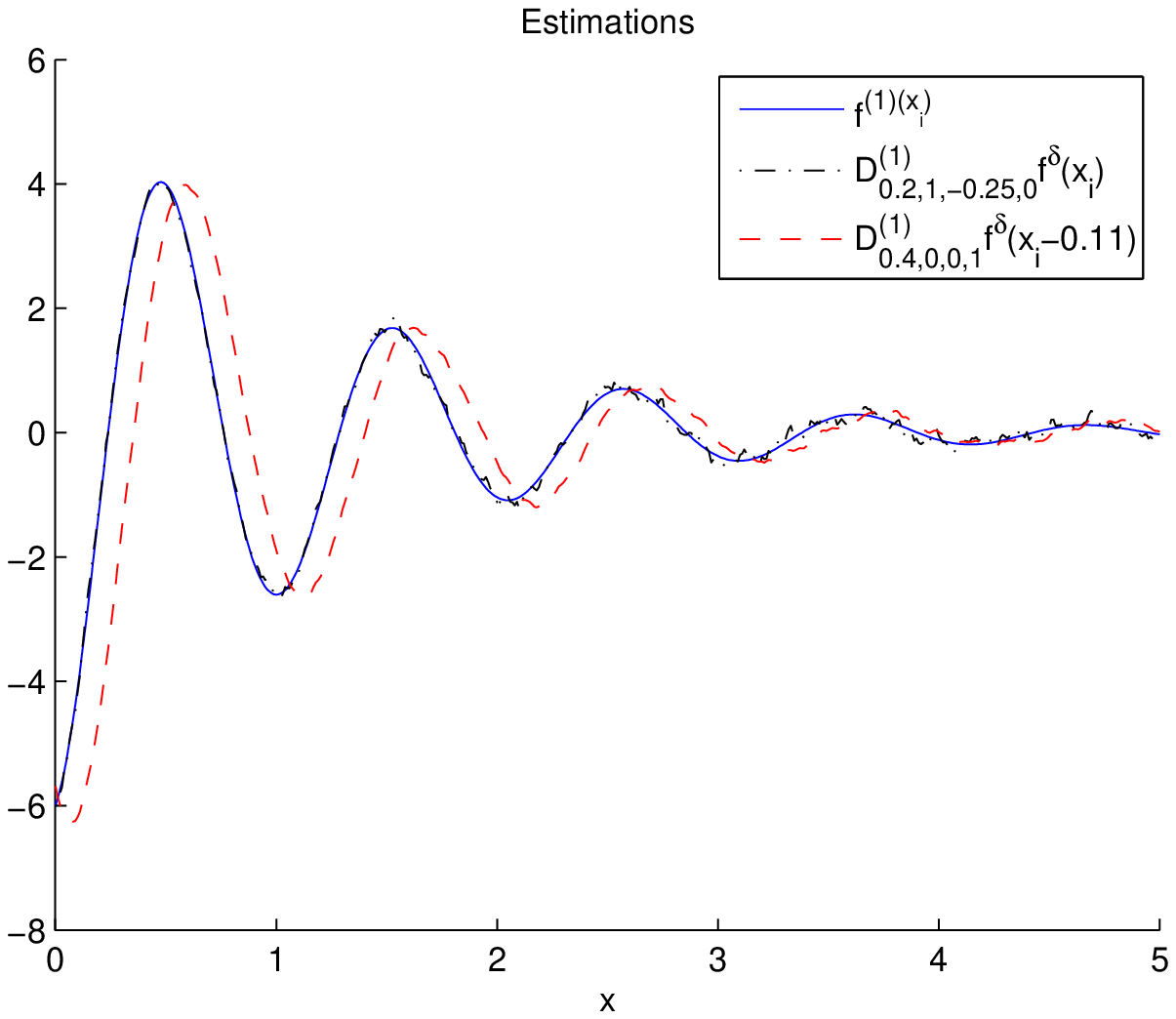}}\caption{Estimations
by Jacobi estimators $D_{h,\alpha,\beta,q}^{(1)}f^{\delta}(x_{i})$
with $h=0.2, \beta=-0.25, \alpha=1, q=0$ and
$D_{h,\alpha,\beta,q}^{(1)}f^{\delta
}(x_{i}-\theta_{2}h)$ with $h=0.4, \alpha=\beta=0, q=1, \theta_{2}=0.276$.}%
\label{simulation0}%
\vskip 2pc
\centering {\includegraphics[scale=0.6]{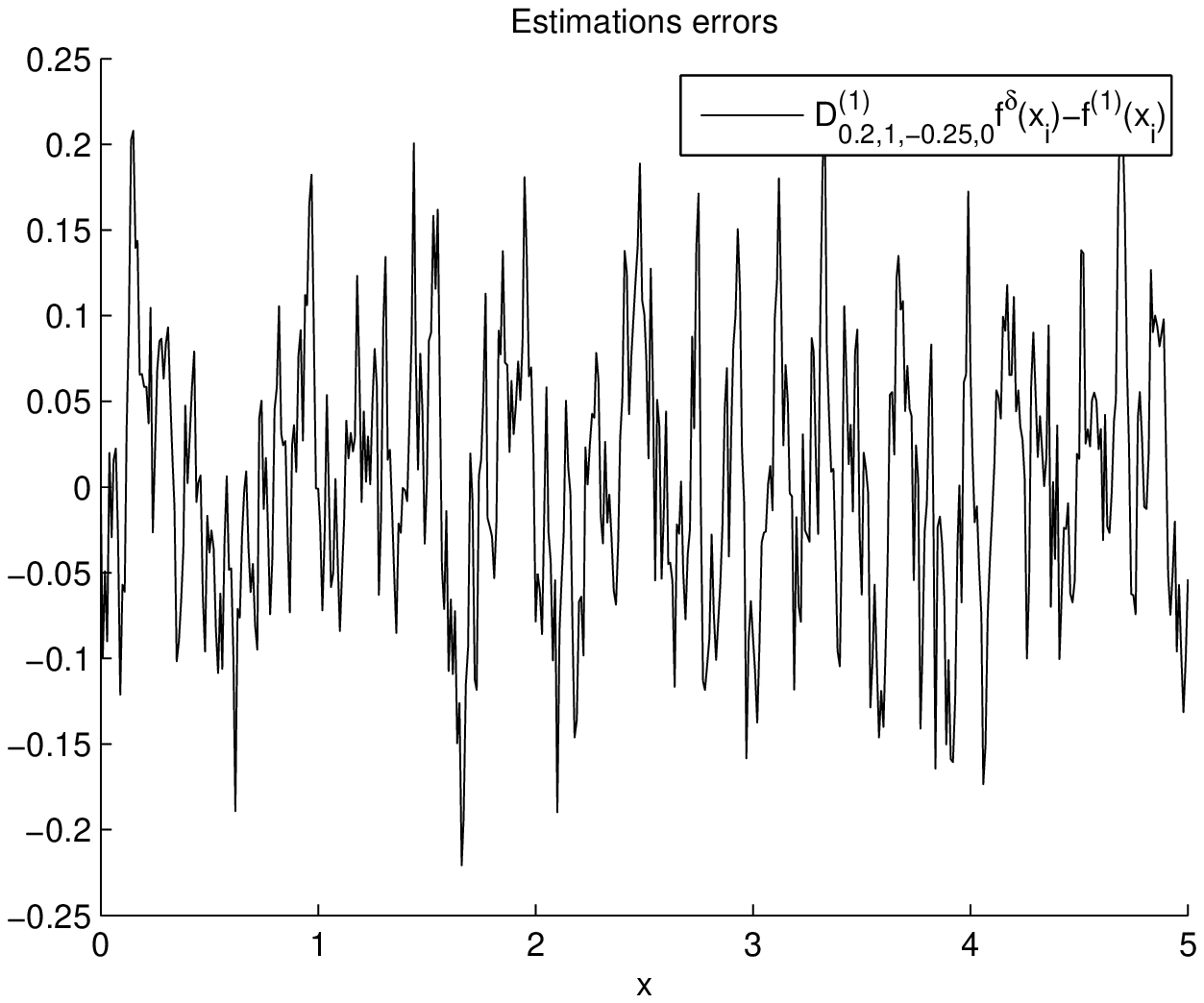}}%
\caption{$D_{h,\alpha,\beta,q}^{(1)}f^{\delta}(x_{i})-f^{(1)}(x_{i})$
with
$h=0.2, \beta=-0.25, \alpha=1, q=0$.}%
\label{simulation1}%
\end{figure}

\begin{figure}[t!]
\centering {\includegraphics[scale=0.6]{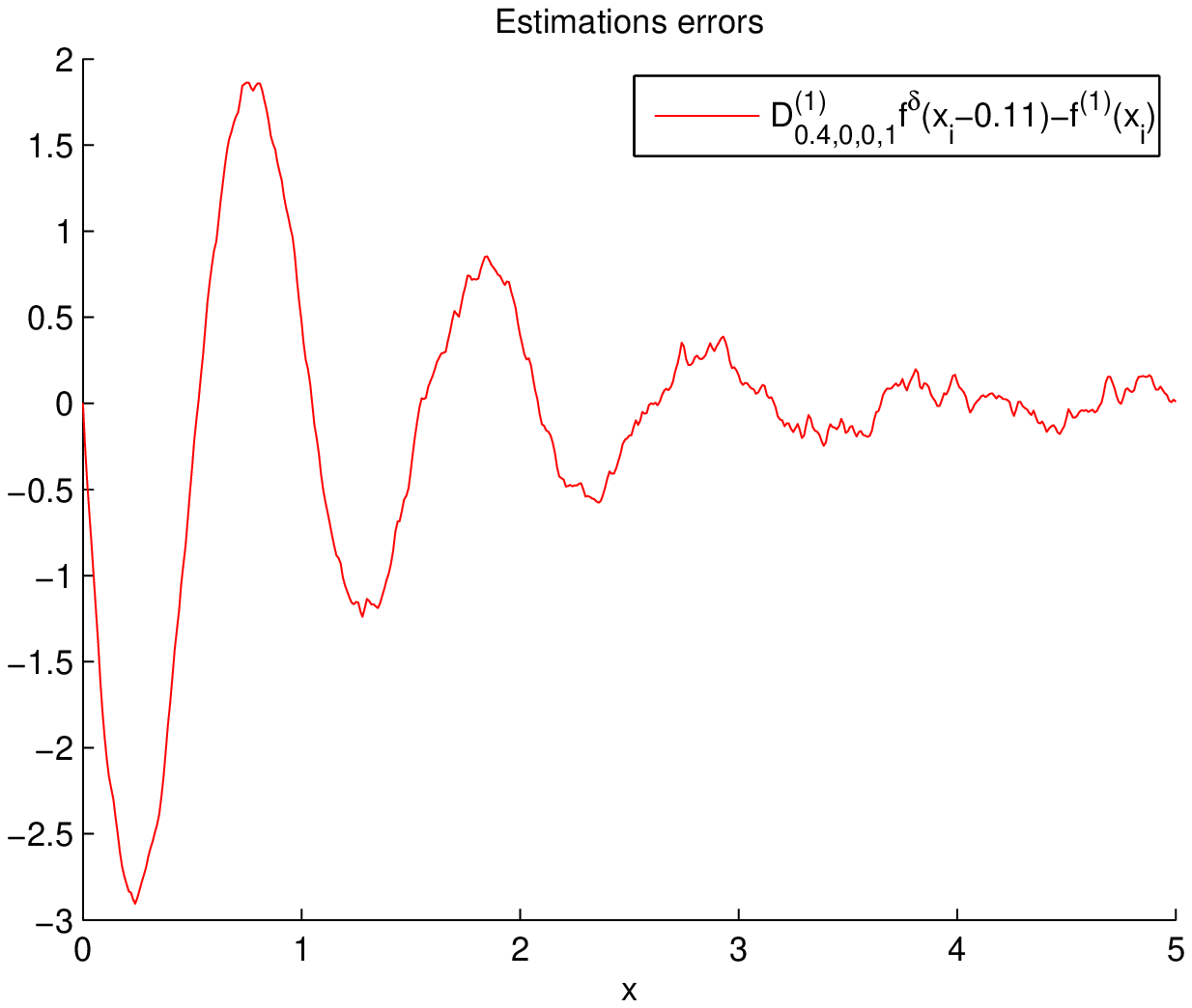}}%
\caption{$D_{h,\alpha,\beta,q}^{(1)}f^{\delta}(x_{i}-\theta_{2}h)-f^{(1)}%
(x_{i})$ with $h=0.4, \alpha=\beta=0, q=1, \theta_{2}=0.276$.}%
\label{simulation2}%

\centering {\includegraphics[scale=0.6]{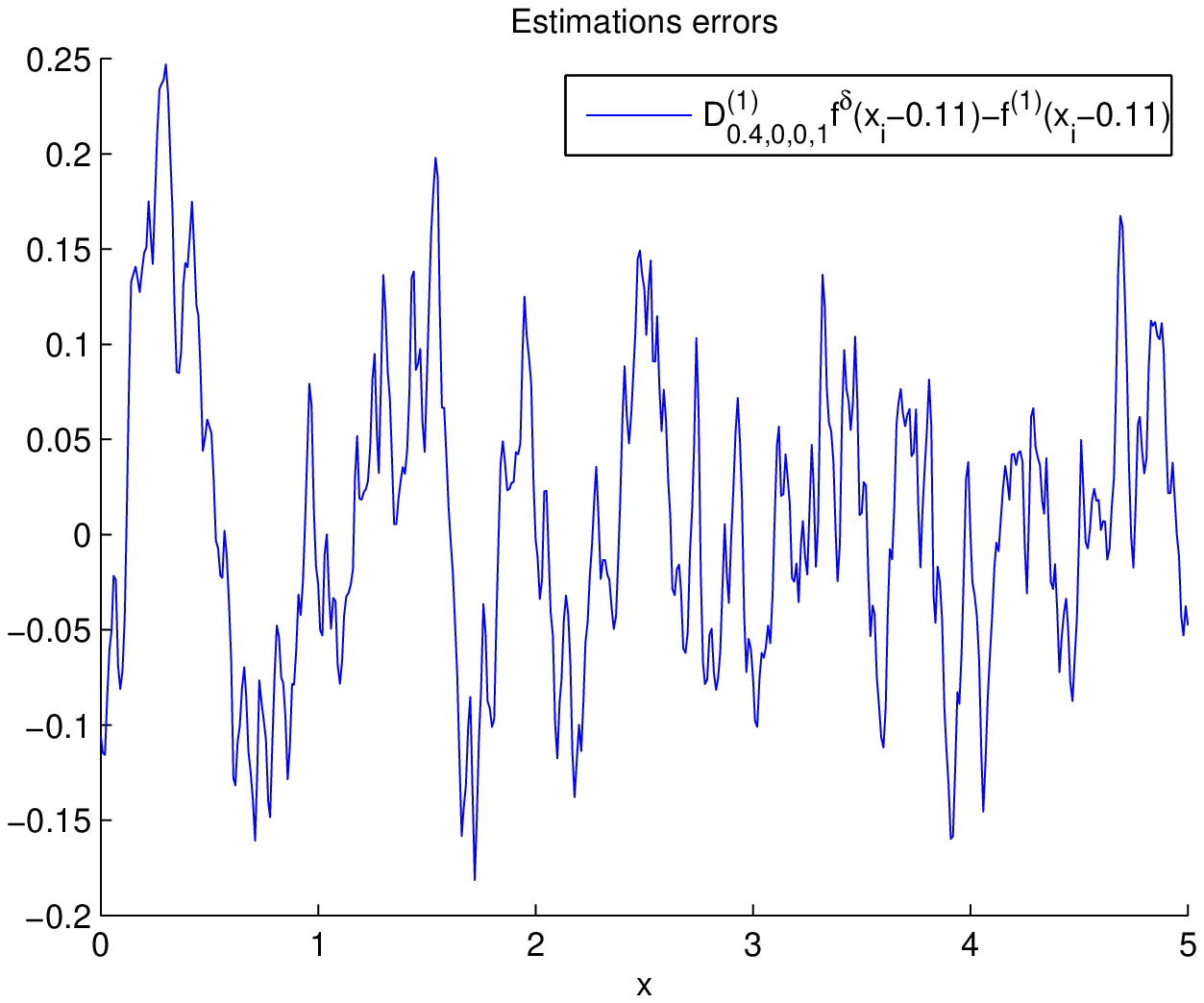}}%
\caption{$D_{h,\alpha,\beta,q}^{(1)}f^{\delta}(x_{i}-\theta_{2}h)-f^{(1)}%
(x_{i}-\theta_{2}h)$ with $h=0.4, \alpha=\beta=0, q=1, \theta_{2}=0.276$.}%
\label{simulation3}%
\end{figure}

\end{document}